\documentclass{amsart}
\usepackage{a4}
\usepackage{multicol}
\usepackage{cite}
\usepackage{amsmath}
\usepackage{amssymb}
\usepackage[all]{xy}
\usepackage{color}
\usepackage{tikz}
\usepackage{hyperref}
\usepackage{enumitem}

\newtheorem{theorem}{Theorem}[section]
\newtheorem{lemma}[theorem]{Lemma}
\newtheorem{proposition}[theorem]{Proposition}
\newtheorem{corollary}[theorem]{Corollary}

\newcommand{\thistheoremname}{}
\newtheorem*{genericthm*}{\thistheoremname}
\newenvironment{namedthm*}[1]
  {\renewcommand{\thistheoremname}{#1}%
   \begin{genericthm*}}
  {\end{genericthm*}}

\theoremstyle{definition}

\newtheorem{definition}[theorem]{Definition}
\newtheorem{example}[theorem]{Example}
\newtheorem{rem}[theorem]{Remark}

\newtheorem{problem}[theorem]{Problem}

\newtheorem{construction}[theorem]{Construction}

\newtheorem{remark}[theorem]{Remark}
\theoremstyle{remark}

\newcommand{\In}{\mathrm{In}}

\newcommand\CC{{\mathbb C}}
\newcommand\KK{{\mathbb K}}

\newcommand\NN{{\mathbb N}}
\newcommand\ZZ{{\mathbb Z}}
\newcommand\RR{{\mathbb R}}
\newcommand\QQ{{\mathbb Q}}
\newcommand\PP{{\mathbb P}}
\newcommand\CO{{\mathcal O}}
\newcommand\A{{\mathbb A}}
\newcommand\F{{\mathcal F}}
\DeclareMathOperator{\Gr}{Gr}
\DeclareMathOperator{\Sym}{Sym}
\DeclareMathOperator{\Char}{char}

\newcommand\D{{\mathcal D}}
\DeclareMathOperator{\Proj}{Proj}
\DeclareMathOperator{\Spec}{Spec}
\DeclareMathOperator{\trop}{Trop}
\DeclareMathOperator{\cone}{cone}
\DeclareMathOperator{\lcm}{lcm}

\DeclareMathOperator{\codim}{codim}
\DeclareMathOperator{\vol}{vol}
\DeclareMathOperator{\ord}{ord}
\DeclareMathOperator{\WDIV}{Div}
\newcommand\bangle[1]{\langle #1 \rangle}

\numberwithin{equation}{section}

\title{Khovanskii-Finite Valuations, Rational Curves, and Torus Actions}
\author{Nathan Ilten}
\address{Department of Mathematics, 8888 University Drive, Simon Fraser University, Burnaby BC V5A1S6}
\email{nilten@sfu.ca}

\author{Milena Wrobel}
\address{Department of Mathematics, 8888 University Drive, Simon Fraser University, Burnaby BC V5A1S6}
\email{milena\_wrobel@sfu.ca}

\begin{document}
\maketitle

\begin{abstract}
We study full rank homogeneous valuations on (multi)-graded domains and ask when they have finite Khovanskii bases. We show that there is a natural reduction from multigraded to simply graded domains. As special cases, we consider projective coordinate rings of rational curves, and almost toric varieties. Our results relate to several problems posed by Kaveh and Manon, and imply that the procedure of Bossinger-Lamboglia-Mincheva-Mohammadi for producing tropical prime cones will not terminate in general.
\end{abstract}

\section{Introduction}
\subsection{Setting and Motivation}
Let $\KK$ be an algebraically closed field, and $X=\Spec R$ an irreducible affine variety over $\KK$. 
We will be studying valuations $\nu$ of $R$ which are trivial on $\KK$, see \S\ref{sec:prelim}. The image of $R$ under such a valuation forms a semigroup $S(R,\nu)$ known as the \emph{value semigroup}. We say that a valuation $\nu$ is \emph{Khovanskii-finite} (with respect to $R$) if $S(R,\nu)$ is finitely generated as a semigroup.\footnote{This differs slightly from terminology of \cite{kaveh-manon}, see our Remark \ref{rem:kf}.}

It is an important question as to when $\nu$ is Khovanskii-finite. For example, when $\nu$ is Khovanskii-finite and has full rank, the variety $X$ admits a flat degeneration to the toric variety associated to $S(R,\nu)$, see e.g.~\cite[Proposition 5.1]{anderson}. This may be used to obtain valuable geometric information concerning $X$. Taking a slightly different point of view, for fixed $R$ it is interesting to try to produce Khovanskii-finite valuations $\nu$.

When $X$ admits an action by an algebraic torus $T$, or equivalently, $R$ is graded by the character lattice $M$ of $T$, it is natural to consider only \emph{homogeneous} valuations, that is, valuations which are completely determined by their values on homogeneous elements of $R$, see Definition \ref{defn:homog}. 
For example, when $X$ is the affine cone over an embedded projective variety $Y$, homogeneous valuations deliver information about $Y$ together with its polarization.

When we are in the situation that $R$ is $M$-graded, we will say that $R$ (or equivalently $X$) is homogeneously Khovanskii-finite if every full rank homogeneous valuation $\nu$ is Khovanskii-finite. An affine (not-necessarily-normal) toric variety $X$ is homogeneously Khovanskii-finite for essentially trivial reasons, see Example \ref{ex:toric}. N.~Ilten and C.~Manon have shown that when $X$ is a normal rational variety with a faithful action by a codimension-one torus, it is also homogeneously Khovanskii-finite \cite[Theorem 5.4]{poised}.

\subsection{Results}
We now highlight our main results. We are interested in criteria when a valuation is Khovanskii-finite, or similarly, when an $M$-graded domain $R$ is homogeneously Khovanskii-finite. Our first main result shows that when considering Khovanskii-finiteness for homogeneous valuations, it is always possible to reduce to the $\ZZ$-graded case:
\begin{theorem}[See Theorem \ref{thm:khofin}]\label{thm:main1}
Let $\nu$ be a full rank homogeneous valuation of an $M$-graded domain $R$.
Assume that the closure $C$ of the cone generated by $S(R,\nu)$ is polyhedral, and let $\pi$ denote the natural projection from $C$ to $M_\RR$.
Then $\nu$ is Khovanskii-finite with respect to $R$ if and only if for 
every ray $\rho\prec C$,
 $\nu$ is Khovanskii-finite with respect to the $\ZZ$-graded domain 
\[
	R_{\pi(\rho)}=\bigoplus_{u\in M\cap \pi(\rho)} R_{u}
\]
and $\pi^{-1}(\pi(\rho))$ is the closure of the cone generated by $S(R_{\pi(\rho)},\nu)$.
\end{theorem}

An \emph{almost toric variety} is a rational not-necessarily-normal variety with a faithful action by a codimension-one torus, see e.g.~\cite{bolin}. In contrast to the result of Ilten and Manon cited above, such varieties are typically \emph{not} homogeneously Khovanskii-finite:
\begin{theorem}[See Theorem \ref{thm:khofincharacterization}]\label{thm:main2}
Assume that $\KK$ is uncountable and $\Char(\KK)=0$. An affine almost toric variety $X=\Spec R$ with $R_0=\KK$ is homogeneously Khovanskii-finite if and only if there exists some $\lambda\in \NN$ such that the Veronese subring
\[
\bigoplus_{u\in M} R_{\lambda\cdot u}
\]
is normal.
\end{theorem}

The analysis of Khovanskii-finiteness for almost toric varieties follows from the reduction in Theorem \ref{thm:main1} together with an analysis of Khovanskii-finiteness for affine cones over projective rational curves. In this setting, we obtain the following results:
\begin{theorem}[See Corollary \ref{cor:thm3}]\label{thm:main3}
Let $X$ be the affine cone over a projective integral rational curve $Y$ and assume that $\KK$ is uncountable and $\Char(\KK)=0$. Then $X$ is homogeneously Khovanskii-finite if and only if $Y$ is smooth.
\end{theorem}
\begin{theorem}[(See Corollary \ref{cor:thm4}]\label{thm:main4}
	Assume $\Char(\KK)=0$.
Let $X=\Spec R$ be the affine cone over a very general integral rational projective plane curve of degree $d>3$. Then there is no homogeneous valuation $\nu$ which is Khovanskii-finite with respect to $R$. 
\end{theorem}

\subsection{Connections to Tropical Geometry}

Given $X$ and $R$ as above, one may seek to construct a full rank valuation $\nu$ which is Khovanskii-finite. 
One approach, due to K.~Kaveh and C.~Manon, is via tropical geometry \cite{kaveh-manon}. After fixing an affine embedding $X\hookrightarrow \A^n$ and intersecting with the torus $(\KK^*)^n$, the tropicalization map yields a polyhedral fan $\trop(X^\circ)$ whose cones $\sigma$ correspond to certain initial degenerations $\In_\sigma(I_X)$ of the ideal $I_X$ of $X$. Maximal cones $\sigma$ with prime initial ideal $\In_\sigma(I_X)$ give rise to  Khovanskii-finite full rank valuations \cite[Theorem 1]{kaveh-manon}, and in some sense all `reasonable' Khovanskii-finite valuations may be obtained in this fashion \cite[Theorem 2]{kaveh-manon}. 

Kaveh and Manon ask \cite[Problem 1]{kaveh-manon} if every projective variety $Y$ may be embedded in such a way that the resulting tropicalization contains a maximal cone with prime initial ideal; this would guarantee the existence of a full rank Khovanskii-finite valuation for $X$, the affine cone over some projective embedding of $Y$. L.~Bossinger, S.~Lamboglia, K.~Mincheva, and F.~Mohammadi propose an algorithm \cite[Procedure 1]{boss-et-al} to produce such an embedding. The algorithm starts with some embedding of $Y$ in projective space, and then proceeds to repair the embedding by essentially taking Veronese re-embeddings coupled with projective changes of coordinates.

Our Theorem \ref{thm:main4} shows that there are examples for which this algorithm will never terminate, see Remark \ref{rem:procedure}. In particular, to solve Problem 1 from \cite{kaveh-manon}, in general one will be forced to alter the polarization of $Y$ in a non-trivial fashion, i.e. one cannot simply take a multiple. In a similar vein, Kaveh and Manon ask \cite[Problem 2]{kaveh-manon} for an algorithm to construct a full rank Khovanskii-finite valuation $\nu$ for a given domain $R$. Our Theorem \ref{thm:main4} again shows that this is impossible in general if one imposes the additional constraint that the valuation be $\ZZ$-homogeneous.

\subsection{Organization}
We now describe the organization of the remainder of the paper. \S\ref{sec:prelim} introduces the central concepts surrounding (homogeneous) valuations and Khovanskii-finiteness. \S\ref{sec:veronese} studies the relationship between Khovanskii-finiteness and Veronese subalgebras and proves our first main result, Theorem \ref{thm:main1}. In \S\ref{sec:norm} we show that passing to the normalization of a domain doesn't change Khovanskii-finiteness properties. \S\ref{sec:global} contains a discussion of a sufficient criterion for the valuation cone of a homogeneous valuation to be polyhedral, an important prerequisite for Khovanskii-finiteness.

In the remaining two sections, we turn our attention to special cases. In \S\ref{sec:curves} we discuss projective coordinate rings for rational curves and prove Theorems \ref{thm:main3} and \ref{thm:main4}. Finally, \S\ref{sec:c1} concerns itself with almost toric varieties, where we prove Theorem \ref{thm:main2} along with other related results.

\subsection*{Acknowledgements}
We thank Bernd Sturmfels for useful discussions, and Sara Lamboglia, Ahmad Mokhtar, and the anonymous referee for helpful comments. Both authors were partially supported by NSERC.

\section{Valuations and Gradings}\label{sec:val}
\subsection{Preliminaries}\label{sec:prelim}

Let $R$ be a finitely generated $\KK$-domain
of Krull dimension $d$, and $(\Gamma, >)$
a totally ordered finitely generated free abelian group. We call 
a map $\nu \colon R\setminus \left\{0\right\} \rightarrow \Gamma$
a \emph{$\KK$-valuation} with values in $\Gamma$ 
if
$\nu(\KK^*) = 0$ holds and for all $f, g \in R \setminus \left\{0\right\}$ we have
$\nu(fg) = \nu(f)+\nu(g)$ and
$\nu(f+g) \geq \min\left\{\nu(f), \nu(g)\right\}$.
The image $S(R,\nu)$ of $\nu$ in $\Gamma$
is a subsemigroup;
the so-called \emph{value semigroup}.
The \emph{valuation cone} $C(R, \nu)$ of $R$ with respect to $\nu$
is the closure in $\Gamma_\RR=\Gamma\otimes \RR$ of the
convex cone
generated by
$S(R, \nu)$. 
The convex cone
generated by $S(R, \nu)$
is not closed in general.
For details on valuations, see e.g.~\cite[VI]{zs}.

The \emph{rank} $r(\nu)$ of a valuation $\nu$ is the rank of the lattice generated by $S(R, \nu)$. 
By possibly shrinking $\Gamma$, we may (and will) always assume that this lattice is $\Gamma$.
The rank $r(\nu)$ is always at most the Krull
dimension of $R$ and we call $\nu$ a \emph{full rank valuation}
if equality holds.

\begin{remark}\label{rem:extend}
Let $R$ be a finitely generated
$\KK$-domain and $\nu$ a valuation.
Then one may extend $\nu$ to the field of fractions $\KK(R)$ of $R$ via $\nu(f/g)=\nu(f)-\nu(g)$.
In particular if $R'$ is a subalgebra of the field of fractions
of $R$, the valuation $\nu$ gives rise to a unique valuation on $R'$
and in the following we denote this valuation also with $\nu$.
\end{remark}

We will be particularly interested in \emph{homogeneous} valuations. For this, let $M$ be a lattice and  
$R$ an $M$-graded finitely generated $\KK$-domain.
\begin{definition}
The \emph{weight monoid} of $R$ is the set
$$S(R):=\left\{w \in M; \ R_w \neq \left\{0\right\}\right\}$$ 
and we call the cone $\omega(R) := \cone(S(R)) \subseteq M_\RR=M\otimes \RR$
the \emph{weight cone} of $R$. 
\end{definition}
\noindent When $R$ is finitely generated, this cone is always rational and polyhedral.
The lattice generated by $S(R)$ is a sublattice of $M$; by replacing $M$ by it, we may (and will) always assume that $S(R)$ generates $M$.

\begin{definition}[{See \cite[\S5.2]{poised}}]\label{defn:homog}
 We call a valuation
$\nu$ as above \emph{$M$-homogeneous} if for all
$f \in R$ 
$$\nu(f) = \min\left\{\nu(f_u)\right\}$$
holds, where $f= \sum_{u \in M} f_u$ is the decomposition
of $f$ into homogeneous elements. 
Moreover we call $\nu$
\emph{fully homogeneous} if $\Gamma=M\times\Gamma'$ and 
for each $u\in M$ and $f_u\in R_u$, $\pi(\nu(f_u)) = u$, where $\pi:\Gamma\to M$ denotes the projection.
\end{definition}
\begin{example}\label{ex:quartic}
Consider the ring
$R:= \KK[tx, t^2x, (1+t^4)x]$
with grading defined by 
$\deg(f(t)x^a) = a$. The weight monoid is $\ZZ_{\geq 0}$, and the weight cone is $\RR_{\geq 0}$.

We consider the homogeneous valuation given by
\[\nu \colon R\setminus\{0\} \rightarrow \ZZ^2, \quad f(t)x^a \mapsto (a, \deg(f(t))).\]
Here we have ordered $\ZZ^2$ lexicographically, with the order on the second factor reversed.
A straightforward calculation shows that for $k>1$, the degree $k$ piece of $R$ has a $\KK$-basis consisting of
\[
	t^2x^k,\ldots,t^{4k-4}x^k,(t+t^{4k-3})x^k,t^{4k-2}x^k,(1+t^{4k})x^k.
\]
It follows that the value semigroup $S(R, \nu)$ is minimally generated 
by the infinite set
$\left\{((1,4), (1,2), (1,1), (k,2)\ | \ \text{ for } k > 2 \right\}$.
\begin{center}
\begin{tikzpicture}[thick,scale=0.3, every node/.style={scale=0.6}]

 \draw[thick] (0,0)--(2, 8); 
 \draw[thick] (0,0)--(5.1,0); 

\foreach \Point in {
(0,0), (1,1), 
(1, 2), (1,4),
(2,2), (2,3), (2,4),(2,5), (2,6), (2,8), 
(3,2), (3,3), (3,4), (3,5), (3,6), (3,7), (3,8),
(4,2), (4,3), (4,4), (4,5), (4,6), (4,7), (4,8),
(5,2), (5,3), (5,4), (5,5), (5,6), (5,7), (5,8)
}{
    \node at \Point {\textbullet};
}

\foreach \Point in {(1,0),(1,3), (2,0), (3,0), (1,1), (2,1), (2,7), (3,1) , (4,0), (4,1), (5,0), (5,1)}{
    \node at \Point {$\circ$};
}
\end{tikzpicture}
\end{center}
The cone generated by this semigroup is not closed. The closure is
$C(R, \nu) = \cone((1,0), (1, 4))$.

We will later see in Example \ref{ex:notkhofin} that for any full rank homogeneous valuation of this domain $R$, the associated semigroup is not finitely generated.
\end{example}

\begin{rem}
Given a fully homogeneous valuation  
\[\nu \colon R\setminus \left\{0\right\} \rightarrow M \times \Gamma',\]
the projection to $\Gamma'$ determines a valuation $\nu'$ on the field of homogeneous fractions of $R$. On the other hand, the image of $S(R,\nu)$ under $\pi$ is exactly $S(R)$. Likewise, the image of $C(R,\nu)$ under $\pi$ is $\omega(R)$.
\end{rem}

\begin{rem}
We may obtain any full rank homogeneous valuation from a fully homogeneous one by composing with an isomorphism of semigroups. Hence, when studying the value semigroups of full rank homogeneous valuations, we will restrict our attention to fully homogeneous ones.
\end{rem}

We now define the key concepts we will be studying:
\begin{definition}
Let $R$ be a $\KK$-domain. 
\begin{enumerate}
\item
We call a valuation $\nu$
\emph{Khovanskii-finite} (with respect to $R$) 
if the value semigroup $S(R, \nu)$ is finitely generated. 
\item
Assume $R$ is $M$-graded. We call
$R$ \emph{homogeneously Khovanskii-finite} 
if every full rank homogeneous valuation $\nu$ 
is Khovanskii-finite. 
\end{enumerate}
\end{definition}

\begin{example}\label{ex:toric}
	Let $R$ be the coordinate ring of an affine (not-necessarily-normal) toric variety $X$. Then $R=\KK[S]$ holds for some finitely generated semigroup $S$. Any full rank homogeneous valuation $\nu$ must have the form 
\[
\nu(\chi^u)=\psi(u),\qquad u\in S
\]
for some injective linear map $\psi:M\to \Gamma$. It follows that the value semigroup $S(R,\nu)$ is isomorphic to $S$, in particular it is finitely generated. Hence, $R$ is homogeneously Khovanskii-finite.
\end{example}

\begin{example}[Example \ref{ex:quartic} continued]
	From Example \ref{ex:quartic} we see that 
	\[
		R=\KK[tx,t^2x,(1+t^4)x]
	\]
	is \emph{not} homogeneously Khovanskii-finite.
\end{example}

\begin{rem}\label{rem:kf}
	Our terminology ``Khovanskii-finite'' arises from the fact that for full rank valuations, Khovanskii-finiteness is equivalent to the existence of a finite Khovanskii basis, see e.g.~\cite[remark on pg.~293]{kaveh-manon}.
We note, however, that for non-full rank valuations, Khovanskii-finiteness (in our sense) does not necessarily imply the existence of a finite Khovanskii basis.
\end{rem}
The following lemma gives a useful criterion for Khovanskii-finiteness:
\begin{lemma}\label{lem:fingen}
Let $S$ be a sub-semigroup of $\ZZ^n$. Then $S$ is finitely generated if and only if the convex cone $\sigma\subset \RR^n$ generated by $S$ is rational polyhedral and closed.
\end{lemma}
\begin{proof}
Assume $\sigma$ is closed and rational polyhedral.
Then for each ray $\rho \preceq \sigma$ there exists an element
$u_{\rho} \in S$ that generates $\rho$. 
Consider the lattice $M$ generated by $S$; this is in fact generated by finitely many elements of $S$, say $u_1,\ldots,u_k$. Let $S'$ be the semigroup generated by $u_1,\ldots,u_k$ and $u_\rho$ as $\rho$ ranges over the rays of $\sigma$.

The saturation of $S'$ in $M$ is $\sigma\cap M$, which in particular contains $S$. Since $\KK[S']$ is noetherian, $\KK[\sigma\cap M]$ is a finitely-generated $\KK[S']$-module, hence so is $\KK[S]$. It follows that $S$ is finitely generated as an $S'$-module, in particular, it is finitely generated as a semigroup. 

The other direction of the lemma is clear.
\end{proof}

We now recall the connection between $S(R,\nu)$, Newton-Okounkov bodies, and degree.
Let $R$ be a $\ZZ_{\geq 0}$-graded finitely generated $\KK$-domain
and $\nu\colon R\setminus \left\{0\right\} \rightarrow \ZZ \times \Gamma$
a homogeneous full rank valuation.
The \emph{Newton-Okounkov body} is the convex set
$$\Delta(R, \nu) := \pi_2(C(R, \nu)\cap \pi_1^{-1}(1)),$$
where $\pi_1, \pi_2$ are the projections of
$(\ZZ\times \Gamma)_\RR $ to $\ZZ_\RR$ and $\Gamma_\RR$.

\begin{remark}[See {\cite[Theorem 2.31]{kk}}] \label{rem:okbo}
Let $Y \subseteq \PP^n$ be a projective variety of dimension $d$ 
and let $R$ be its homogeneous coordinate ring. Then $R$ comes with a natural
$\ZZ_{\geq 0}$-grading. Consider a homogeneous full rank valuation $\nu$.
Then the dimension of the Newton-Okounkov body $\Delta(R, \nu)$ 
equals the degree of the Hilbert polynomial. Its volume
is the leading coefficient of the Hilbert polynomial, and thus
$1/\dim(Y)!$ times the degree of $Y$. 
\end{remark}

\subsection{Veronese Subalgebras}\label{sec:veronese}
In studying Khovanskii-finiteness of a domain $R$, one natural operation at our disposal is that of passing to a \emph{Veronese subalgebra}:

\begin{definition}
Let $R$ be a finitely generated $M$-graded $\KK$-domain
and $L \subseteq M$ a submonoid.
The \emph{Veronese subalgebra associated to $L$} is
$$R(L) := \bigoplus_{w \in L} R_w \subseteq R.$$
If $\sigma \subseteq M_\QQ$ is a cone we set for short
$R_\sigma := R(\sigma \cap M)$.
\end{definition}
\noindent If $L$ is a finitely generated submonoid of $M$, then $R(L)$ will also be finitely generated, see e.g.~\cite[Proposition 1.1.2.4]{coxbook}. In particular, If $\sigma$ is a rational polyhedral cone, then $R_\sigma$ is finitely generated.
By Remark \ref{rem:extend}, if $\nu$ is any valuation on $R$, we may also view it as a valuation on any Veronese subalgebra $R(L)$.

Our main result in this section says that to check Khovanskii-finiteness of a fully homogeneous valuation, one may essentially reduce to the $\ZZ$-graded case. Geometrically, this says that instead of considering $\Spec R$, we may consider its GIT quotients under the torus action induced by the grading. 
\begin{theorem}\label{thm:khofin}
Let $R$ be a finitely generated $M$-graded $\KK$-domain
with $R_0 = \KK$
and $\nu \colon R\setminus\{0\} \rightarrow \Gamma=\Gamma'\times M$ a fully homogeneous valuation.
Assume that ${C(R, \nu)}$ is rational polyhedral, and let $\pi:\Gamma_\RR\to M_\RR$ be the projection.
Then the following statements are equivalent:
\begin{enumerate}
\item\label{it:1}
$\nu$ is Khovanskii-finite with respect to $R$.
\item \label{it:2}
For all rational rays $\rho\subseteq{C(R, \nu)}$, 
$\nu$ is Khovanskii-finite with respect to
the Veronese subalgebra $R_{\pi(\rho)}$ and the valuation cone of $R_{\pi(\rho)}$ 
equals $\pi^{-1}(\pi(\rho))$.
\item \label{it:3}
 For all rays $\rho\preceq{C(R, \nu)}$,
$\nu$ is Khovanskii-finite with respect to
the Veronese subalgebra $R_{\pi(\rho)}$ and the valuation cone of $R_{\pi(\rho)}$ 
equals $\pi^{-1}(\pi(\rho))$.
\end{enumerate}
\end{theorem}

\begin{proof}
Assume $\nu$ is Khovanskii-finite with respect to $R$. Then
$S(R, \nu)$ is finitely generated and thus 
for every  ray $\rho\preceq C(R, \nu)$
there exists an element 
$f_\rho \in R_{\pi(\rho)}$ 
such that $\nu(f_\rho)$ generates $\rho$. By considering products of such elements, the same claim follows for every rational ray $\rho\subseteq C(R,\nu)$.

For any such ray $\rho$, we have 
$$S(R_{\pi(\rho)}, \nu) = S(R, \nu) \cap \pi^{-1}(\pi(\rho)).$$
Since $S(R, \nu)$ is finitely generated, so is $S(R, \nu) \cap \pi^{-1}(\pi(\rho))$.
This shows that (\ref{it:1}) implies (\ref{it:2}). Condition (\ref{it:3}) follows trivially from (\ref{it:2}).

Suppose finally that condition (\ref{it:3}) holds and consider a ray $\rho$ of ${C(R, \nu)}$. Since the valuation $\nu$
is Khovanskii-finite  with respect to
$R_{\pi(\rho)}$
and
$${C(R_{\pi(\rho)}, \nu)} 
=
\pi^{-1}(\pi(\rho) )
$$
holds,
there exists an element 
$f_\rho \in R_{\pi(\rho)}$ 
such that $\nu(f_\rho)$ generates $\rho$. 
The first condition now follows from Lemma \ref{lem:fingen}.
\end{proof}
\noindent For an illustration of this theorem at work in a special case, see Example \ref{ex:at}.

We will next show that the second condition in Theorem \ref{thm:khofin} can be weakened somewhat. 
For Khovanskii-finiteness, it is enough to check
the condition $\pi^{-1}(\pi(\rho)) = C(R_{\pi(\rho)}, \nu)$ only for the rays
$\rho$ with $\pi(\rho) \cap \omega(R)^\circ = \emptyset$:
\begin{lemma}\label{lem:VSA}
Let $R$ be a finitely generated $M$-graded $\KK$-domain
with $R_0 = \KK$ and 
$$\nu\colon R \setminus \left\{0\right\} \rightarrow M \times \Gamma$$ 
a fully homogeneous valuation. 
Consider rational hyperplanes $H_1, \ldots, H_r \subseteq M_\RR$
such that 
$H_1 \cap\ldots\cap  H_r \cap \omega(R)^\circ \neq \emptyset$. Then for
$R' := R(H_1\cap \ldots \cap H_r \cap S(R))$ we have
$${C(R', \nu)} 
=
\pi^{-1}(\omega(R'))\cap {C(R, \nu)},
$$
where $\pi \colon C(R, \nu) \rightarrow M_\RR$ denotes the projection.
\begin{proof}
Let $S(R)\subseteq M$ be the weight monoid and 
$\omega(R)$ the weight cone of $R$. 
Let $a_1, \ldots, a_r$ be generators for 
$\omega(R) \cap M$. Then
$\lambda_i a_i \in S(R)$ holds for suitable
$\lambda_i \in \ZZ_{\geq 1}$, since $S(R)$ is pointed and generates $M$. 
Setting $\lambda := \lcm(\lambda_1, \ldots, \lambda_r)$
we obtain a saturated submonoid
$S' := S(R) \cap \lambda M \subseteq S(R)$ and the Veronese subalgebra $R(S')$ of $R$ has the same valuation cone as $R$. 
So in the following considerations we may restrict to the case that $S(R)$ is saturated. 
Using $M \cong \ZZ^m$ and $\Gamma \cong \ZZ^n$
for some $m,n \in \ZZ_{\geq 0}$
we may furthermore assume 
$\nu\colon R \setminus \left\{0\right\} \rightarrow \ZZ^{m+n}$.

In a first step we show the statement for Veronese subalgebras $R'$ defined
by an intersection of one hyperplane with the weight monoid.
Let $H\subseteq \QQ^m $ be any hyperplane 
with $H \cap \omega(R)^\circ \neq \emptyset$
and fix elements $v \in \ZZ^m, u^+, u^- \in S(R, \nu)$ such that $v$ is primitive and
$\bangle{v, H} = 0$, $\bangle{v, \pi(u^+)} = 1$ and
$\bangle{v, \pi(u^-)} = -1$. 
These exist since $S(R)$ is saturated.
We have to show that
for any point $x$ in the relative boundary of 
$$
\pi^{-1}(\omega(R'))\cap {C(R, \nu))} 
$$
and any $\epsilon>0$ there is a point 
$x' \in S(R', \nu)$ and
$\mu \in \QQ$ with
$|x-\mu x'| < \epsilon$.

Let $a \neq 0$ be any point in the relative boundary. 
As $R_0 = \KK$ and $\nu(\KK^*) = 0$ holds we may assume $a:=(1, a_2, \ldots, a_{m+n}) \in \QQ_{\geq 0}^{m+n}$ with $\pi(a)=(1, a_2, \ldots, a_m) \in \omega(R)$ the corresponding weight. Fix $\epsilon' >0$. Then there exists a point 
$(k, a_2', \ldots, a_{m+n}') \in S(R, \nu)$ such that
$|a_i'- k \cdot a_i| < k \cdot \epsilon'$ holds for all $i$. 
Let $\pi(a')=(k, a'_2, \ldots, a'_m) \in \omega(R)$ be the corresponding weight.

Consider $\lambda:=-\langle v,\pi(a')\rangle$; if this is zero, then $\pi(a')\in H$ and we are done. So without loss of generality, we assume $\lambda >0$. Then
$\pi(a') + \lambda \pi(u^+) \in H\cap S(R)$.
We obtain
\begin{align*}
0 \leq \lambda &= \bangle{-v, \pi(a')} = \bangle{-v, \pi(a')-k\pi(a)} 
\\
&\leq
\left|\sum_{i=2}^m v_i(a_i'-ka_i)\right| \leq \sum_{i=2}^m |v_i||a_i'-ka_i| < \left(\sum_{i=2}^m |v_i|\right)k\epsilon'
\end{align*}
and thus
\begin{align*}
\left|\frac{1}{k + \lambda u_1^+}(a_i'+\lambda u_i^+ - a_ik - a_i\lambda u_1^+)\right| 
& \leq
\left|\frac{a_i'- a_ik}{k + \lambda u_1^+}\right| +
\left|\frac{\lambda u_i^+ }{k + \lambda u_1^+}\right| + 
\left|\frac{a_i\lambda u_1^+}{k + \lambda u_1^+}\right|
\\
 &<
 \epsilon' + u_i^+(\sum |v_i|)\epsilon' + a_iu_1^+(\sum |v_i|) \epsilon'. 
\end{align*}
As $a, v$ and $u^+$ are fixed we conclude
$|\frac{1}{k + \lambda u_1^+}(a'+\lambda u^+) -a| < \epsilon$
for suitably chosen~$\epsilon'$.
This proves the claim for a single hyperplane.

By using this argument inductively we obtain the statement for Veronese subalgebras 
$R'$ defined by the intersection of several rational hyperplanes $H_1, \ldots, H_r$ with the weight cone such that $H_1\cap \ldots \cap H_r \cap \omega(R)^\circ \neq \emptyset$.
\end{proof}
\end{lemma}

\subsection{Normalization}\label{sec:norm}
The second natural operation at our disposal is that of passing to the \emph{normalization} $\bar R$ of $R$, that is, its integral closure in its field of fractions.  
In general, it is not true that normalization and taking Veronese subalgebras commute. However, there are some special circumstances where this is true.

\begin{lemma}\label{lem:raygen}
Let $R$ be an $M$-graded $\KK$-domain.
\begin{enumerate}
	\item If $M'$ is a finite index sublattice of $M$, then
		\[
			\overline{R(M')}=\bar{R}(M').
		\]
	\item Suppose $R$ has simplicial weight cone $\omega(R)$, and is generated by elements
whose weights lie on the rays of $\omega(R)$.
Then for any ray $\rho \preceq \omega(R)$,
\[\overline{R(\rho \cap M)} = \bar{R}(\rho \cap M).\]
\end{enumerate}
\begin{proof}
For the first claim, let 
$a \in \bar R(M')$  be a homogeneous element of degree $u\in M'$.
Then $a=a_1/a_2$ for homogeneous $a_i\in R$, whose difference of degrees lies in $M'$, and after multiplying by an appropriate element of $R$, we may assume $a_i\in R(M')$. The element $a$ satisfies an integral equation
$$a^k = c_{k-1}a^{k-1} + \ldots + c_0$$
for homogeneous $c_i \in R$ of degree $(k-i) \cdot u$, so $c_i\in R(M')$. It follows that $\bar R(M') \subset \overline{R(M')}$. The opposite inclusion is straightforward.

For the second claim, we proceed similarly with $a \in \bar R(\rho \cap M)$  homogeneous of degree $u$. 
As the degrees of the generators of $R$ 
lie on the rays of the simplicial cone $\omega(R)$ 
we conclude
$a = a_1/a_2$ for $a_i\in R$ where 
$\deg(a_i) \in \rho$.
Moreover
$$a^k = c_{k-1}a^{k-1} + \ldots + c_0$$
for homogeneous $c_i \in R$ of degree $(k-i) \cdot u$.
In particular $a_1, a_2, c_i \in R(\rho \cap M)$ and thus $a \in \overline{R(\rho \cap M)}$ holds. 
The opposite inclusion is again straightforward.
\end{proof}
\end{lemma}

\begin{lemma}\label{lem:degequ}
Let $R$ be a finitely generated $\ZZ$-graded $\KK$-domain 
such that $R_0 = \KK$ holds and let $\bar R$ denote its normalization.
Assume $R$ and $\bar R$ are generated in degree one
and set 
$X:= \Proj(R)$ and
$\bar X := \Proj(\bar R)$.
Then $ \deg(X) = \deg(\bar X)$ holds.
\begin{proof}
We have the following commutative diagram
$$ 
\xymatrix{
{\Proj (\bar R)}
\ar@{}[r]|=
&
{\bar{X}}
\ar@{}[r]|\subseteq
\ar[d]_{p}
& 
{\PP^n}
\ar[d]^{\pi}
\\
{\Proj (R)}
\ar@{}[r]|=
&
X
\ar@{}[r]|\subseteq
&
{\PP^m}
}
$$
where $p \colon \bar X \rightarrow X$ denotes the
normalization map and $\pi$ is a linear projection. 
Let $L$ be a $\codim(X)$-dimensional 
linear subspace of $\PP^m$ in general position. 
Then $\deg X = X\cdot L$ and 
$\deg \bar X = \bar X \cdot (\pi^*L)$ holds and
we conclude
$$\deg \bar X 
= \bar X \cdot (\pi^*L) 
= (\pi_*\bar X) \cdot L 
= (\deg(p)\pi(\bar X)) \cdot L
= X\cdot L = \deg X$$
by the projection formula.
\end{proof}
\end{lemma}

We are now able to show that the valuation cone does not change under normalization:
\begin{proposition}\label{prop:norm}
Let $R$ be an $M$-graded $\KK$-domain with 
$R_0 = \KK$,
normalization $\bar R$ 
and let $\nu$ be a homogeneous full rank valuation. Then
we have
$${C(R, \nu)} = {C(\bar R, \nu)}.$$
\end{proposition}

\begin{proof}
By fixing a positive $\ZZ$-grading given by a projection
$\lambda \colon M_\QQ\rightarrow \ZZ_\QQ$ we reduce to the case that $M=\ZZ$.
Fixing minimal sets of generators of $R$ and $\bar R$, and letting $\lambda$ be the least common multiple of the degrees of these generators, we replace $R$ by the Veronese subalgebra $R(\lambda\ZZ)$. After rescaling the grading, this is generated in degree one, and by Lemma \ref{lem:raygen} we obtain $\overline{R(\lambda\ZZ)}=\bar{R}(\lambda\ZZ)$. This is also generated in degree one. Furthermore,
${C(R, \nu)} = {C(R(\lambda\ZZ), \nu)}$ and a similar statement holds for $\bar R$. Thus, we may assume that $R$ and $\bar R$ are generated in degree one.
Then the corresponding projective varieties $X$ and $\bar X$
share the same degree due to 
Lemma~\ref{lem:degequ}.

We may assume that the $\ZZ$-homogeneous valuation $\nu$ is
of the form
$$\nu \colon R \setminus \left\{0\right\} \rightarrow \ZZ \times \Gamma, \quad f_k \rightarrow (k, \nu'(f_k)),$$
where $f_k \in R_k$ for $k \in\ZZ_{\geq 0}$ and $\nu'\colon R \rightarrow \Gamma$ is a valuation.
Denote by $\pi_1$ the projection of $(\ZZ \times \Gamma)_\RR$ to $\ZZ_\RR$.
Using \cite[Theorem 2.31]{kk}
(see Remark \ref{rem:okbo}) we obtain
$$
\vol({C(R, \nu)} \cap \pi_1^{-1}(1)) 
= \vol(\Delta(R, \nu))
=  \vol(\Delta(\bar R, \nu))
= \vol({C(\bar R, \nu)} \cap \pi_1^{-1}(1)).
$$
The desired equality follows since $S(R, \nu) \subseteq S(\bar R, \nu)$ holds.
\end{proof}

\begin{example}[Example \ref{ex:quartic} continued]
Consider the ring $R$ as in Example \ref{ex:quartic}.
Then its normalization is given as $\bar R = \KK[x, xt, xt^2, xt^3, xt^4]$.
In particular the value semigroup 
$S(\bar R, \nu_\infty)$ is the saturated semigroup
 of the cone $\cone((1,0), (1,4))={C(R, \nu_\infty)}$.
\end{example}

\subsection{When is $C(R,\nu)$ rational polyhedral?}\label{sec:global}
	If ${C(R, \nu)}$ is not rational polyhedral, then $S(R,\nu)$ cannot be finitely generated, that is, $R$ is not Khovanskii-finite with respect to~$\nu$. It thus becomes interesting to obtain a criterion for $C(R,\nu)$ to be rational polyhedral.
	By Proposition \ref{prop:norm}, we may replace a non-normal $R$ by $\bar R$ without changing $C(R,\nu)$, so in the following we will assume that $R$ is normal. 
	We will give a sufficient condition in this case for $C(R,\nu)$ to be rational polyhedral.

Assume that $R$ is normal with $R_0=\KK$.
	In this case, $R$ may be constructed as 
\[
	R=\bigoplus_{u\in \omega(R)\cap M} H^0(Y,\CO(\lfloor(\D(u)\rfloor)))\cdot \chi^u
\]
where $Y$ is a normal projective variety and $\D\colon\omega\cap M_\QQ\to \WDIV_\QQ(Y)$ is a so-called p-divisor, 
see \cite{pdiv}.
This means that $\D$ is a piecewise linear concave map taking values in semiample $\QQ$-Cartier divisors, with $\D(u)$ big for $u$ in the interior of $\omega$. Furthermore, $\D$ has only finitely many regions of linearity.
Above, $\lfloor \D(u) \rfloor$ denotes the integral round-down of the $\QQ$-divisor $\D(u)$.

A distinguished class of full rank valuations on $\KK(Y)$ comes from a choice of a full flag $\F_0\subset\F_1\subset\ldots\subset\F_d=Y$ of irreducible subvarieties of $Y$,
 where the point $\mathcal{F}_0$ is a smooth point in each $\mathcal{F}_i$  
(see \cite[\S 1.1]{LM}, \cite[Example 2.13]{kk}).
  For $f \in \KK(Y)$ one computes the value $\nu_{\mathcal{F}}(f) = (a_1, \ldots, a_d) \in \ZZ^d$ recursively, where  $\ZZ^d$ is endowed with the lexicographic ordering.  The first component $a_1$ is the order of vanishing of $f$ along the divisor $\mathcal{F}_{d-1}$.  If $s$ is a local equation for $\mathcal{F}_{d-1}$ at $p$, $s^{-a_1}f$ can be regarded as a non-zero rational function on $\mathcal{F}_{d-1}$. This allows the procedure to be repeated with the divisor $\mathcal{F}_{d-2} \subset \mathcal{F}_{d-1}$ to produce $a_2$, and so on until the process terminates with $a_d$, the order of vanishing at $\F_0$.   

  The valuation $\nu_{\mathcal{F}}$ extends naturally to give a full rank homogeneous valuation on $R$, which we also denote by $\nu_{\mathcal{F}}$:
\begin{align*}
	\nu_{\mathcal{F}}:R\setminus\{0\}&\to M\times \ZZ^d\\
	f\cdot \chi^u&\mapsto(u,\nu_{\mathcal{F}}(f))\qquad f\in \KK(Y).
\end{align*}

Given a divisor $D$ on $Y$, one may consider the section ring 
\[
	S(D):=\bigoplus_{k\in {\ZZ_{\geq 0}}} H^0(Y,\CO(kD )) \chi^k.	
\]	
This is equipped with a natural homogeneous valuation $\nu_{\F,D}$ \cite{LM}:
if $f_D$ is a local equation for $D$ at the point $\F_0$, for a section $s\in H^0(X,\CO(kD))$ we have
\[
\nu_{\F,D}(s)=(k,\nu_\F(s)+k\nu_\F(f_D))\in\ZZ\times\ZZ^d.
\]
Let $N^1_\RR(Y)$ denote the N\'eron-Severi space of $Y$.
The \emph{global Newton-Okounkov body} $\Delta(Y,\nu_{\mathcal{F}})$ is a closed convex cone in $N^1_\RR(Y)\times \RR^d$ such that for any big divisor $D$, the fiber of $\Delta(Y,\nu_{\mathcal{F}})$ over the ray of $N^1_\RR$ generated by the class $[D]$  of $D$ is the image of $C(S(D),\nu_{\mathcal{F},D})$ under the inclusion induced by $(1,a)\mapsto ([D],a)$, see \cite[Theorem B]{LM} for  precise details.

\begin{theorem}\label{thm:global}
	Let $R$ be a normal $M$-graded domain with $R=\KK_0$ and $Y,\mathcal{F}$ as above.
	If the global Newton-Okounkov body $\Delta(Y,\nu_{\mathcal{F}})$ is rational polyhedral, then so is $C(R,\nu_{\mathcal{F}})$.
\end{theorem}
\begin{proof}
	Set $\nu=\nu_{\mathcal{F}}$.
We consider the linear extension of the map $\D$ to $\D:\omega\cap M_\RR\to \WDIV_\RR Y$. The weight cone $\omega(R)$ decomposes into finitely many regions of linearity $\omega_i$, each of which is still a rational polyhedral cone. Since $C(R,\nu)$ is the closure of the convex hull of the $C(R_{\omega_i},\nu)$, it suffices to show that each $C(R_{\omega_i},\nu)$ is rational polyhedral. Hence, we reduce to the case that $\D$ is linear on $\omega$.
	
We have a commutative diagram 
$$ 
\xymatrix{
{M_\RR\times\RR^d}
\ar[r]^\phi
\ar[d]_{\pi}
&
{N^1_\RR(Y)\times\RR^d}
\ar[d]_{p}
\\
{M_\RR}
\ar[r]^\psi
&
{N^1_\RR(Y)}
}
$$
whose maps we now describe. The map $\psi$ sends $u$ to the class $[\D(u)]$. 
The two vertical arrows $\pi,p$ are just projections.
	For $u\in M$, let $c(u)= \nu(f_{\lambda \D(u)})/\lambda $, where $\lambda\in \NN$ is such that $\lambda \D(u)$ is Cartier, and $f_{\lambda\D(u)}$ is a local equation for $\lambda\D(u)$ at $\mathcal{F}_0$. This extends  to a linear map $c:M_\RR\to \RR^d$, which we use to construct
the linear map $\phi(u,a)=(\psi(u),a+c(u))$.

	Let $\sigma$ denote the closure of the cone of big divisor classes in $N^1_\RR(Y)$. 
Then $\psi(\omega)$ is contained in $\sigma$. 	
By the construction of $\Delta(Y,\nu)$ and the description of $\phi$ we have that for any rational ray $\rho$ in $\omega$,
\begin{equation}\label{eqn:global}
C(R_\rho,\nu)\subseteq \{(u,a)\in \rho\times\RR^d\ |\  \phi(a)\in p^{-1}(\psi(u))\cap \Delta(Y,\nu)\}
\end{equation}
with equality if $\rho$ is in the interior of $\omega$. 
If  $\Delta(Y,\nu)$ is polyhedral, so is 
\[
\{(u,a)\in \sigma\times\RR^d\ |\ \phi(a)\in p^{-1}(\psi(u))\cap \Delta(Y,\nu) \}
\]
and by \eqref{eqn:global} this must equal $C(R,\nu)$.
\end{proof}

\section{Rational Curves}\label{sec:curves}
As a simplest case of the $\ZZ$-graded domains arising in Theorem \ref{thm:khofin}, we will consider in this section finitely generated $\ZZ$-graded domains $R$ with $R_0=\KK$, and with the field of homogeneous fractions of $R$ a rational function field. Geometrically, this means that we are considering $\ZZ$-graded domains $R$ such that $\Proj(R)$ is a (possibly singular) rational curve. Such curves have been studied from a different point of view in e.g.~\cite{curves1}.

Since we are interested in Khovanskii-finiteness properties of $R$, we may always replace $R$ with a Veronese subalgebra $R(\lambda \ZZ)$ for some $\lambda\in \NN$ to ensure that $R$ is generated in degree one. By considering the inclusion of $R$ in its normalization $\bar R$, it is straightforward to show that all such domains may be constructed as follows:
\begin{construction}
Consider an $m$-dimensional linear subspace 
$L \subseteq \mathrm{Sym}^{a}(\KK^2) \cong \KK[s,t]_a$
and define
$$R(L) := \bigoplus_{k \geq 0} L^k \subseteq \bigoplus_{k \geq 0} \mathrm{Sym}^{ka} (\KK^2),$$
where $L^k$ denotes the linear span of elements of the form $l_1 \cdots l_k$ for $l_i\in L$. We endow $R(L)$ with its natural $\KK$-algebra structure. 
\end{construction}

\noindent The space
$\Sym^a(\KK^2)$ has dimension $a+1$ and thus $m \leq a+1$.
Moreover, if $m =1$ then $R(L) \cong \KK[x]$ 
and if $m=2$ then $R(L) \cong \KK[x,y]$.
Henceforth, for any further considerations we will assume $m\geq 3$.

\begin{definition}
Let $L \subseteq \mathrm{Sym}^{a}(\KK^2)$. 
The \emph{degree} of $L$ 
is the leading coefficient of the Hilbert polynomial of $R(L)$
i.e. the degree of $\Proj \, R(L)$. 
The \emph{arithmetic genus} of $L$ is 
$1-a_0$, where $a_0$ is the constant coefficient
of the Hilbert polynomial, i.e.,
it is the arithmetic genus of $\Proj \, R(L)$.
\end{definition}

Setting $Y=\Proj R(L)$, the arithmetic genus of $L$ equals the arithmetic genus of $Y$. In particular, $Y$ is smooth if and only if $L$ has arithmetic genus $0$.
Recall that a point of a scheme is \emph{very general} if it lies in the complement of a countable union of closed subschemes.

\begin{lemma}\label{lem:vg}
Let  $L \subseteq \mathrm{Sym}^{a}(\KK^2)$ be
a very general linear subspace of dimension $m$, 
where $3 \leq m \leq a+1$.
Then the degree of $L$ equals $a$. 
\begin{proof}
For each $k$, the condition $\dim(L^k) < j$  is a closed condition in 
$\mathrm{Gr}(m, \mathrm{Sym}^a(\KK^2))$.  Let $Z_{k,j}$ denote the corresponding subvariety and set
\[
Z_k(d,g):=\bigcap_{i\in\ZZ_{\geq 0}} Z_{k+i,d(k+i)-g+1}. 
\]
The locus
 in 
$\Gr(m, \mathrm{Sym}^a(\KK^2))$
such that $\deg(L) < d$ is the countable union of the $Z_k(d,g)$ as $k,g$ vary.

To conclude the proof we 
show that for every $m$ there 
exists an $m$-dimensional linear subspace
$L$ of degree $a$.
Consider the following $m$-dimensional 
linear subspace
$$L := \bangle{s^a, s^{a-1}t, \dots, s^{a-m+2}t^{m-2}, t^a} \subseteq \KK[s,t]_a.$$
A straightforward computation shows that $R(L)$ has degree $a$
and the assertion follows.
\end{proof}
\end{lemma}

\begin{remark}\label{rem:genus}
If $L$ is a $3$-dimensional linear subspace of $\Sym^{a}(\KK^2)$, 
then $\Proj R(L)$ is a rational plane curve and
hence has arithmetic genus
$$g= \frac{(d-1)(d-2)}{2},$$
where $d$ is the degree of $L$.

Furthermore, since $R(L)\cong \KK[x,y,z]/f$ for a degree $d$ form $f$,
we have ${\dim (L^k)=k\cdot d+1-g}$ as long as $k\geq d$. This means that $d=a$ if and only if $L$ is not in the variety $Z_{d,d^2-g+1}$ from the proof of Lemma \ref{lem:vg}. In particular, any general $L$ in $\Gr(3,\Sym^{a}(\KK^2))$ has degree $d$ and genus $g$.
\end{remark}

Let $L \subseteq \Sym^a(\KK^2)$ be a linear subspace 
and $Q \in \PP^1$ any point. Then the following assignment defines a 
homogeneous full rank valuation on $R(L)$:
$$\nu_Q\colon R(L)\setminus\{0\} \rightarrow \ZZ^2, \quad f \mapsto (k, \ord_Q(f)) \ \text{ for } f \in L^k,$$
where $\ZZ^2$ is ordered lexicographically. Any full rank homogeneous valuation on $R(L)$ will have value semigroup isomorphic to $S(R(L),\nu_Q)$ for some $Q\in \PP^1$ \cite[\S 5.2]{poised}, so in the following we may restrict our attention to homogeneous valuations of  the form~$\nu_Q$.

\begin{theorem}\label{thm:curve}
Let $L \subseteq \Sym^d(\KK^2)$ be a linear subspace of 
degree $d$. 
Then for $Q=(\alpha:\beta)\in\PP^1$, $\nu_Q$ is
Khovanskii-finite with respect to $R(L)$ 
if and only if there exists $k>0$ such that $L^k$ contains $(\beta s - \alpha t)^{dk}$.
\end{theorem}
\begin{proof}
Since $R(L)$ has degree $d$ and $\nu_Q$ is of full rank,
the valuation cone $C(R, \nu_Q)$ equals 
$\cone((1,0), (1,d))$, see Remark \ref{rem:okbo}.
Suppose that $\nu_Q$ is Khovanskii-finite with respect to $R(L)$. 
Then by Lemma \ref{lem:fingen} the value semigroup
$S(R, \nu_Q)$ contains ray generators
$(k_1, 0), (k_2, k_2d)$  
for $k_1, k_2 \in \ZZ_{\geq0}$.
Setting $k := k_1\cdot k_2$
we obtain $(k,0), (k, kd) \in S(R, \nu)$. 
Thus there exists $f \in L^k$ such that 
$\nu_Q(f) =(k, kd)$, that is, 
$f$ is divisible by $(\beta s -\alpha t)^{dk}$, 
which implies $(\beta s-\alpha t)^{dk} \in L^k$. 

Conversely suppose that $(\beta s-\alpha t)^{dk}\in L^k$ holds. 
Then $(k, dk) \in S(R, \nu_{(\alpha, \beta)})$. 
Now let $b_1$ be the smallest integer such that $(1, b_1) \in S$. 
Then, since $R(L)$ is generated in degree one,
$(k, b_2) \in S$ implies $b_2\geq k b_1$. 
As $\deg(R(L)) = d$ holds it follows that $b_1 =0$. 
By Lemma \ref{lem:fingen}
this implies that $R(L)$ is Khovanskii-finite with 
respect to~$\nu_Q$. 
\end{proof}

\begin{remark}\label{rem:curve}
Viewing $\PP^1$ as $\PP(\KK[s,t]_1)$ and 
$\PP^{kd}$ as $\PP(\KK[s,t]_{kd})$, 
we consider the curve $C_{kd}$ as the image of
$$\kappa \colon \PP^1 \rightarrow \PP^{kd}, \quad (\beta s-\alpha t) \mapsto
(\beta s-\alpha t)^{kd}.$$

We may rephrase the above theorem in the following way:
the valuation $\nu_Q$
 is Khovanskii-finite with respect to
$R(L)$  if and only if there exists $k>0$ such that 
$L^k\subseteq \Sym^{kd}(\KK^2)$ contains the image of $Q$ in $C_{kd}$.
We conclude that $R(L)$ admits a Khovanskii-finite homogeneous valuation if and only if $L^k\subseteq \Sym^{kd}(\KK^2)$ intersects $C_{kd}$ non-trivially for some $k$.

If $\Char(\KK)=0$, then $C_{kd}$ is projectively equivalent to the rational normal curve of degree $kd$; in particular, it is not contained in any proper linear subspace. 
\end{remark}

\begin{example}\label{ex:finite}
Assume $\Char(\KK)=0$. Consider the $3$-dimensional linear subspace $L$ of $\Sym^3(\KK^2)$ generated by 
\[
s^3,s^2t,t^3.
\]
For each $k> 0$, $L^k=R(L)_k$ has the basis
\[
	s^{3k},s^{3k-1}t,\ldots,s^2t^{3k-2},t^{3k}.
\]
Then by looking at the coefficient of $st^{3k-1}$, we observe that $(\beta s -\alpha t)^{3k}\in L^k$ if and only if $\beta\alpha^{3k-1}=0$, that is, $\alpha=0$ or $\beta=0$. From Theorem \ref{thm:curve}
we infer that $\nu_Q$ is Khovanskii-finite with respect to $R=R(L)$ if and only if $Q=(1:0)$ or $Q=(0:1)$.

Observe that $R$ is actually a \emph{toric} algebra: it is isomorphic to the semigroup algebra $\KK[S]$, where $S$ is generated by $(0,1),(1,1),(3,1)$. So despite being homogeneously Khovanskii-finite with respect to this $\ZZ^2$-grading, it is not homogeneously Khovanskii-finite with respect to the $\ZZ$-grading from above (where $\deg s=\deg t=1$).
\end{example}

\begin{example}\label{ex:notkhofin}
Assume again that $\Char(\KK)=0$.
	For $d\geq 4$, consider the $(d-1)$-dimensional linear subspace $L$
of $\Sym^d(\KK^2)$ generated by 
\[s^{d-1}t, \ldots, s^{2}t^{d-2}, s^d+t^d\]
For $k>1$, $R(L)_k$ may be described as follows:

\noindent
{\bf{Case $d=4$}} (see also Example \ref{ex:quartic}): $R(L)_k$ has a basis 
$$
s^{kd-2}t^2, \ldots, s^4t^{kd-4}, s^3t^{kd-3} + s^{kd-1}t, s^2t^{kd-2}, s^{kd}+t^{kd}$$

\noindent
{\bf{Case $d \geq 5$}}: $R(L)_k$ has a basis 
$$s^{kd-1}t, \ldots, s^2t^{kd-2}, s^{kd} +t^{kd}.$$
We immediately see that in all cases, $L$ has degree $d$.

Applying Theorem \ref{thm:curve} we obtain that there 
exists no homogeneous valuation that is Khovanskii-finite
with respect to $R(L)$.
Indeed 
$(\beta s -\alpha t)^{dk}$ has either a non-zero $st^{kd-1}$ component if $\alpha \neq 0$  or equals $\beta^{dk}t^{dk}$.
In particular,
for every $k >0$ we obtain
$L^k \cap C_{dk} = \emptyset$,
where $C_{dk}$ denotes the curve from Remark \ref{rem:curve}.

Consider the linear subspace
$L' := \bangle{s^{d-1}t, s^{d-2}t^2, s^d+t^d} \subseteq L.$
Then
$$R(L') = \KK[s^{d-1}t, s^{d-2}t^2, s^d+t^d]  \cong \KK[x_1, x_2, x_3]/ \bangle{x_1^d+x_2^d-x_1^{d-2}x_2x_3}$$
holds and thus $R(L')$ is of degree $d$. 
Note that
$L^k \cap C_{dk} = \emptyset$ implies
$(L')^k \cap C_{dk} = \emptyset$ and thus 
there exists no valuation which is Khovanskii-finite
with respect to $R(L')$.
In particular for every integer $d \geq 4$ 
there exists a rational plane curve $C$
 of degree $d$ admitting no Khovanskii-finite homogeneous valuation.
\end{example}

\begin{theorem}\label{thm:genuskhofincurves}
Let $L \subseteq \Sym^d(\KK^2)$ be of degree $d$ and arithmetic genus $g$. 
If $g \leq 1$ then 
there exists a homogeneous valuation that is 
Khovanskii-finite with respect to $R(L)$.
Furthermore, $R(L)$ is homogeneously Khovanskii-finite if $g=0$, and the converse is true as long as $\KK$ is uncountable and $\Char(\KK)=0$.
\begin{proof}
If the genus of $L$ equals one then for $k \gg 0$ we have
$\dim(L^k) = kd = \dim(\mathrm{Sym}^{kd}(\KK^2)) -1$.
In particular $L^k$ has codimension one in 
$\mathrm{Sym}^{kd}(\KK^2)$ and thus intersects the 
curve $C_{kd}$ of Remark \ref{rem:curve} non-trivially. 
Hence Theorem \ref{thm:curve} implies that $R(L)$ has a Khovanskii-finite valuation.
If the genus of $L$ is equal to zero
there exists a $k \gg 0$ such that
the vector space $L^k$  equals 
$\Sym^d(\KK^2)$. In particular 
$L^k$ includes $C_{kd}$ and thus
$R$ is homogeneously Khovanskii-finite.

Suppose instead that $g\geq 1$ and $\Char(\KK)=0$. Then for each $k$, $C_{kd}$ intersects $L^k$ in at most finitely many points, since $C_{kd}$ is non-degenerate. This implies that there are at most countably many $Q$ for which $\nu_Q$ is Khovanskii-finite with respect to $R(L)$. Assuming $\KK$ is uncountable, this shows that $R(L)$ cannot be homogeneously Khovanskii-finite. 
\end{proof}
\end{theorem}

\begin{corollary}\label{cor:thm3}
Let $X$ be the affine cone over a projective integral rational curve $Y$ and assume that $\KK$ is uncountable and $\Char(\KK)=0$. Then $X$ is homogeneously Khovanskii-finite if and only if $Y$ is smooth.
\end{corollary}
\begin{proof}

If $X$ is $\Spec R$, then $R$ is generated in degree one and has degree say $d$. Embedding $R$ into its normalization $\bar R\cong \bigoplus_{k\geq 0}\Sym^{ka}(\KK^2)$ we set $L=R_1$. Then $a=d$ by Proposition \ref{prop:norm} and $R\cong R(L)$. We now apply the above theorem.
\end{proof}
\begin{remark}\label{rem:cor}
	Corollary \ref{cor:thm3} applies to any $X=\Spec R$ for which $\Proj R=Y$, not just those with $R$ generated in degree one. Indeed, since  passing to a standard Veronese subalgebra changes neither the value semigroups nor the curve $Y = \Proj R$, we may assume that $R$ is generated in degree one. 
\end{remark}

\begin{example}\label{ex:countable}
Even if $R$ is not Khovanskii-finite, there may be an infinite number of valuations of the form $\nu_Q$ which are Khovanskii finite. For example, with $\KK=\CC$ consider the domain
\[
	R(L),\qquad L=\bangle{s^2t,st^2,s^3+t^3}.
\]
It is straightforward to show that the degree $k$ piece $L^k$ of $R(L)$ has basis 
\[\{	s^{3k-1}t,s^{3k-2}t^2,\ldots,st^{3k-1},s^{3k}+t^{3k}\}.\]
Using this basis of $L^k$, it follows that the condition $(\beta s-\alpha t)^{3k}\in L^k$ is equivalent to
\[
	\beta^{3k}=(-\alpha)^{3k}.
\]
By Theorem \ref{thm:curve}, $\nu_Q$ is thus Khovanskii-finite with respect to $R$ if and only if $Q=(1:\eta)$ for some root of unity $\eta$. In particular, $R$ is not Khovanskii-finite, but there are infinitely many Khovanskii-finite valuations.
\end{example}

\begin{theorem}\label{thm:kf}
	Assume $\Char(\KK)=0$.
Let $L \subseteq \Sym^d(\KK^2)$ be a very general 
$3$-dimensional linear subspace
for $d\geq 4$.
Then there exists no homogeneous valuation which is
Khovanskii-finite with respect to $R(L)$.
\begin{proof}
Setting 
\[g:=\frac{(d-1)(d-2)}{2},\]
we consider the incidence variety
$$W_k = \left\{(V,P) \in \Gr(kd+1-g, kd+1) \times C_{kd}\ |\  \ P \in V\right\}$$
inside of 
$$\Gr(kd+1-g, kd+1) \times C_{kd}.$$
This incidence variety is closed 
as it is locally given by the vanishing of minors,
and hence projective. So its image $Y_k$ in 
$\Gr(kd+1-g, kd+1)$ is closed.

For $k\geq d$, consider the rational map 
\[
\phi_k: \Gr(3, \Sym^d(\KK^2)) \to \Gr(kd+1-g, kd+1) 
\]
sending $L$ to $L^k$. The indeterminacy locus of this map is contained in a fixed proper subvariety $Z_{d,dk-g+1}$, see Remark \ref{rem:genus}.

Let $Y_k'$ be the preimage in 
$\Gr(3, \Sym^d(\KK^2))$ of $Y_k$ under the rational map $\phi_k$.
By Example \ref{ex:notkhofin} above the inclusion
$\overline{Y_k'}\subset \Gr(3, \Sym^d(\KK^2))$
is proper.
Thus $\Gr(3, \Sym^d(\KK^2))$ properly contains 
$\Delta = \cup_{k \geq 0} \overline{Y_k'}$. 
Choosing $L$
in the complement of $\Delta$ and $Z_{d,dk-g+1}$, it follows that there exists no homogeneous valuation which is
Khovanskii-finite with respect to $R(L)$.
\end{proof}
\end{theorem}

\begin{corollary}\label{cor:thm4}
	Assume $\Char(\KK)=0$.
	Let $X=\Spec R$ be the affine cone over a very general integral rational plane curve of degree $d>3$. Then there is no homogeneous valuation $\nu$ which is Khovanskii-finite with respect to $R$.
\end{corollary}
\begin{proof}
	This is just a geometric reformulation of Theorem \ref{thm:kf}.
\end{proof}

\begin{example}
	While Corollary \ref{cor:thm4} guarantees that the affine cone over a very general rational plane curve of degree at least four has no homogeneous Khovanskii-finite valuations, there do exist \emph{special} rational plane curves of arbitrary degree with homogeneous Khovanskii-finite valuations.
Consider for example the toric variety
$\Proj \KK[s^4,s^3t, t^4]$, which is a plane projective curve of 
degree $4$ and genus $3$.
We consider the valuation $\nu_{Q}$ on  $R:=\KK[s^4,s^3t, t^4]$ for $Q=(0,1)$.
The value semigroup is generated by $(4,1)$, $(3,1)$, and $(0,1)$, in particular it is finitely generated.
\end{example}

\begin{remark}\label{rem:procedure}
	In \cite[Procedure 1]{boss-et-al}, an algorithm is proposed which, given a presentation $\KK[x_0,\ldots,x_n]/I$ of a finitely generated $\KK$-algebra $A$, attempts to compute a new presentation such that the corresponding tropicalization has more so-called prime cones. By \cite{kaveh-manon}, this will lead to a Khovanskii-finite full rank valuation $\nu$ on~$A$.

	The algorithm proceeds by adding certain new generators $y_1,\ldots,y_s$ to the already existing generators $x_0,\ldots,x_n$. In the case that $A$ is $M$-graded and the generators $x_i$ are homogeneous with respect to $M$, then the new generators $y_i$ the algorithm produces will also be homogeneous. This implies that, should the algorithm produce an embedding with a new prime cone, the corresponding valuation will be $M$-homogeneous.

	From this we conclude: if $A$ is $M$-homogeneous, and has no full rank homogeneous valuations which are Khovanskii-finite, then the algorithm \cite[Procedure 1]{boss-et-al} will not terminate if its input is any homogeneous presentation of $A$. In particular, Corollary \ref{cor:thm4} shows that for any very general plane curve of degree at least four in characteristic zero, this algorithm will not terminate.
\end{remark}

\begin{remark}
	One might instead ask about Khovanskii-finiteness properties for affine cones over \emph{smooth, non-rational} curves $Y$. Here, the homogeneous valuations one must consider are not as straightforward as in the rational case.
However, we do remark that every smooth projective curve $Y$ admits an 
embedding such that there exists a homogeneous valuation $\nu$
which is Khovanskii-finite with respect to 
the corresponding homogeneous coordinate ring.
Indeed, 
let $g$ be the genus of $C$, let $P\in C$ be any point
and consider the divisor
$D := (2g+1)P$.
Set $R=\bigoplus_{k\in \ZZ_{\geq 0}} H^0(Y,\CO(kD))$.
Since $\deg D=2g+1$, $D$ is very ample, $R$ is generated in degree one, and $Y=\Proj R$.
We consider the homogeneous valuation $\nu$ sending a section $f$ of $H^0(Y,\CO(kD))$ to $(k,\nu_P(f))\in\ZZ^2$, where $\nu_P$ denotes the order of vanishing at the point $P$.

We consider the valuation $\nu$ sending a section $f$ of $H^0(Y,\CO(kD))$ to $(k,\nu_P(f)$.
As $D$ is basepoint free there exists an $f \in H^0(Y,\mathcal{O}(D))$ with
$\ord_P(f) = -d$,
and as $D-P$ is base point free 
there exists an $f' \in H^0(Y,\mathcal{O}(D))$ 
with $\ord_P(f') = -d+1$.
Note that $1 \in H^0(Y,\mathcal{O}(D))$ holds.
So the value semigroup 
$S= S(R(D), \nu)$ contains
$(1,0), (1, -d)$ and $(1,1-d)$.
Since $\deg(R(D)) = d$ holds
these are ray generators for the valuation cone and thus by Lemma \ref{lem:fingen}, $S$ is finitely generated. 
\end{remark}

	While Theorem \ref{thm:curve} gives a criterion for checking whether the affine cone over a rational curve admits a homogeneous Khovanskii-finite valuation, it is not  effective: one must check countably many conditions to show that the ring $R$ admits no such valuation. It would be interesting to solve the following problem:
	\begin{problem}\label{prob:effective}
	Find an effective algorithm to determine whether an affine cone over a rational curve admits a homogeneous Khovanskii-finite valuation.
\end{problem}
\noindent One might naively hope to show that the degree and genus of a curve give an upper bound on the values of $k$ for which one must check the criterion of Theorem~\ref{thm:curve} for any fixed valuation $\nu_Q$. However, Example \ref{ex:countable} shows that this is impossible. Nonetheless, one could hope for a bound of $k$ in terms of degree and genus such that if $L^k\cap C_{dk}=\emptyset$, then 
$L^{k'}\cap C_{dk'}=\emptyset$ for all $k'>k$.

\section{Almost Toric Varieties}\label{sec:c1}
Recall that an \emph{almost toric variety} is a rational not-necessarily normal variety $X$ equipped with a faithful action by a codimension-one torus. Any affine almost toric variety is parametrized by functions which are products of monomials with univariate rational functions. As such they are a natural generalization of toric varieties, and appear naturally in more applied settings such as algebraic statistics, see e.g. \cite[Equation (18)]{cumulant}.

In this section, we will study Khovanskii-finiteness for almost toric varieties. As a first step, we improve on Theorem \ref{thm:khofin} in this special case. See \S\ref{sec:veronese} for notation on Veronese subalgebras.

\begin{theorem}\label{thm:almtorkhofin}
Let $R$ be a rational finitely generated $M$-graded $\KK$-domain
of complexity one
with $R_0 = \KK$
and $\nu \colon R \rightarrow \Gamma$ a fully homogeneous valuation. 
Let $\pi$ denote the projection $C(R,\nu)\to M_\RR$.
Then $C(R, \nu)$ is rational polyhedral and 
the following statements are equivalent:
\begin{enumerate}
\item
$\nu$ is Khovanskii-finite with respect to $R$.
\item
$\nu$ is Khovanskii-finite with respect to
the Veronese subalgebras $R_{\pi(\rho)}$,
where $\rho$ runs over all rays of ${C(R, \nu)}$.
\end{enumerate}
\end{theorem}

Before proving the theorem, we illustrate with an example:
\begin{example}\label{ex:at}
	Consider the $\ZZ^2$-graded domain
	\[R=K[x,(t-1)x,(t-1)^3x,ty,t^2y,(1+t^3)y]\]
	with $\deg(x) = (1,0)$ and $\deg(y) = (0,1)$.
	This has weight cone $\RR_{\geq 0}^2$.
	Using Theorem \ref{thm:almtorkhofin}, we will see that a full rank homogeneous valuation $\nu$ is Khovanskii-finite with respect to $R$ if and only if the restriction of $\nu$ to $\KK(t)$ is equivalent to the valuation $\nu_1$ with $\nu_1(t-1)=1$.

	In order to apply the theorem, we need to know the rays of the valuation cone.	Following \cite{langlois}, we may describe $\bar R$ as
	\[
		\bigoplus_{u\in \omega\cap M} H^0(\PP^1,\CO(\D(u)))\cdot \chi^u
	\]
	for $\D(u)=3(u_1+u_2)\cdot\{\infty\}\in\WDIV(\PP^1)$,
	where we are viewing $t$ as a local parameter on $\PP^1$ whose principal divisor is $\{0\}-\{\infty\}$, and we are identifying $x$ with $\chi^{(1,0)}$ and $y$ with $\chi^{(0,1)}$. Alternatively, an explicit calculation shows that $\bar R=\KK[x,tx,t^2x,t^3x,y,ty,t^2y,t^3y]$.
	
	However, the previous description in terms of $\D$ is more useful for us. From this description and \cite[Corollary 5.5]{poised}, it follows that for any full rank homogeneous valuation $\nu$, 
	the valuation cone $C(R,\nu)$ is lattice equivalent to 
	\[
		\cone\{(1,0,0),(1,0,3),(0,1,0),(0,1,3)\}\subset M_\RR\times \RR.
	\]
	In particular, any ray of $C(R,\nu)$ projects to the ray $\rho_1$ generated by $(1,0)$ or $\rho_2$ generated by $(0,1)$ in $\omega$.

	Thus, to apply the theorem, we must consider the Veronese subalgebras
	\begin{align*}
		R_{\rho_1}=&\KK[x,(t-1)x,(t-1)^3x]\\
		R_{\rho_2}=&\KK[ty,t^2y,(1+t^3)y]
	\end{align*}
	For $R_{\rho_1}$, we obtain Example \ref{ex:finite} after homogenization and change of coordinates, so $\nu$ is Khovanskii-finite with respect to $R_{\rho_1}$ if and only if the restriction of $\nu$ to $\KK(t)$ is equivalent to $\nu_1$ above, or $\nu_\infty$, the valuation with $\nu_\infty(t)=-1$.

	For $R_{\rho_2}$, we obtain Example \ref{ex:countable} after homogenization and change of coordinates, so $\nu$ is Khovanskii-finite with respect to $R_{\rho_2}$ if and only if the restriction of $\nu$ to $\KK(t)$ is equivalent to $\nu_{\eta}$ for some root of unity $\eta$, where $\nu_{\eta}(t-\eta)=1$. Combining these two criteria, it follows from Theorem \ref{thm:almtorkhofin} that $\nu$ is Khovanskii-finite with respect to $R$ if and only if its restriction to $\KK(t)$ is equivalent to $\nu_1$.
\end{example}

\begin{proof}[Proof of Theorem \ref{thm:almtorkhofin}]
Due to Proposition \ref{prop:norm} we have
${C(R, \nu)} = C(\bar R , \nu)$,
where $\bar R$ denotes the normalization of $R$.
Thus $C(R, \nu)$ is rational polyhedral 
due to \cite[Theorem 1.3]{poised} and Lemma \ref{lem:fingen}.

The second condition follows from the first by Theorem \ref{thm:khofin}. Likewise, to show that the second condition implies the first, we may use Theorem \ref{thm:khofin} and Lemma \ref{lem:VSA}
to reduce to showing that for any ray $\rho$ of $C(R,\nu)$ with $\pi(\rho)$ in the relative boundary of $\omega=\omega(R)$, 
$\pi^{-1}(\pi(\rho)) = C(R_{\pi(\rho)}, \nu)$ holds.
If $\omega$ is one-dimensional, there is nothing to show, so we assume that $\dim(\omega)\geq 2$.
We will prove the claim by a series of reductions. First, we reduce to the case that $\dim (\omega)=2$, then we reduce to the case that the degrees of the generators of $R$ lie on the rays of $\omega$, from which we then reduce to the case that $R$ is normal.

Consider a ray $\rho$ with $\pi(\rho)$ in the relative boundary of $\omega$.
Then there exists
an intersection of hyperplanes $H_1, \ldots, H_r$ 
with $H_1\cap \ldots \cap H_r \cap \omega(R)^\circ \neq \emptyset$ 
such that $H_1 \cap \ldots \cap H_r$ is two-dimensional and 
contains $\pi(\rho)$. In particular $\pi(\rho) \preceq H_1 \cap \ldots \cap H_r \cap \omega(R)$. Using Lemma \ref{lem:VSA}
we thus may reduce to the case that
$\omega$ is two-dimensional.

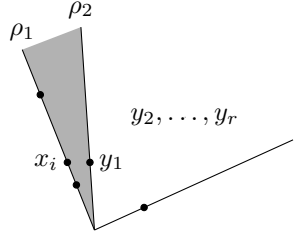
\begin{center}
\begin{tikzpicture}[scale=0.6]
    \path[fill=gray!60!] (0,0)--(-1.6,4)--(-0.3,4.5)--(0,0);
    \draw (0,0)--(-1.6,4) node[above] {$\rho_1$};
    \draw (0,0)  -- (-0.3,4.5) node[above] {$\rho_2$};
    \draw (0,0)  -- (4.4,2);
    \filldraw[black] (-0.1,1.5) circle (2pt) node[right] {$y_1$};
    \filldraw[black] (-0.4,1) circle (2pt);
    \filldraw[black] (-0.6, 1.5) circle (2pt) node[left]{$x_i$};
    \filldraw[black] (-1.2,3) circle (2pt);
    \path[fill, color=black] (2,2.5) circle (0.0ex)  node[]{$y_2, \ldots, y_r$};
    \filldraw[black] (1.1,0.5) circle (2pt);
  \end{tikzpicture}
\end{center}

Set $\rho_1=\pi(\rho)$. Let $x_1,\ldots,x_t,y_1,\ldots,y_r$ be 
be homogeneous 
generators of $R$ 
such that $\deg(x_{i}) \in \rho_1$ holds and $\deg(y_1)$ generates the ray 
$\rho_2$ that is the closest to $\rho_1$ among all rays generated by $\deg(y_i)$. 
Consider the subcone $\omega'$ generated by $\rho_1$ and $\rho_2$
and note that
$\pi^{-1}(\pi(\rho))= \pi^{-1}(\rho_1) \subseteq C(R_{\omega'}, \nu)$
holds.

The subalgebra
$R_{\omega'}$ is finitely generated by homogeneous elements 
$$
x_1, \ldots, x_t, y_1, z_1, \ldots, z_s \quad 
\text{with} 
\quad 
z_j = \prod_{i=1}^t  x_{i}^{a_{ji}} \cdot \prod_{k=1}^r y_k^{b_{jk}}
$$
for some $a_{ji},b_{jk}\in\ZZ_{\geq 0}$.
Furthermore, we may find rational numbers
$\mu_{j} \in \QQ_{\geq 0}$ such that
$\deg(z_j) - \sum_i \mu_{j}a_{ji} \deg(x_{i}) \in \rho_2$ and
$0\leq \mu_{j}\leq 1 $ holds.

Let $\alpha\in \NN$ be such that $\alpha\cdot \mu_j\in \ZZ_{\geq 0}$.
Consider the subalgebra $R'$ of $R_{\sigma'}$ generated by the elements
$ 
x_{1}, \ldots, x_{t}, y_1$ and
\[\hat z_j:=z_j^\alpha /  (\prod_{i=1}^t x_{i}^{\mu_{j}a_{ji} \cdot \alpha})
=\prod_{i=1}^t  x_{i}^{\alpha(1-\mu_j)a_{ji}} \cdot \prod_{k=1}^r  y_k^{\alpha b_{jk}}
 \]
for $j =1, \ldots, s$. The degree of each $\hat z_j$ lies on $\rho_2$, so $R'$ is generated by elements whose degrees lie on the rays of $\omega'$. Moreover,
 $R'$ contains the $\alpha$th Veronese subalgebra of $R_\omega'$, so $C(R', \nu) = C(R_{\omega'}, \nu)$.
Thus, replacing $R$ by $R'$ and $\omega$ by $\omega'$, we may assume that $R$ is generated by elements whose degrees lie on the rays of $\omega$.

In this situation, we may apply Lemma \ref{lem:raygen} and Proposition \ref{prop:norm}  to conclude that $C(R_{\pi(\rho)}, \nu)=C(\bar R_{\pi(\rho)},\nu)$. Hence, we may replace $R$ by $\bar R$, that is, we may assume that $R$ is normal. But by \cite[Theorem 1.3]{poised}
$R$ is homogeneously Khovanskii-finite.
In particular $\pi^{-1}(\pi(\rho)) = C(R_{\pi(\rho)}, \nu)$ holds
for every ray $\rho \preceq C(R, \nu)$.
\end{proof}

\begin{remark}
It is straightforward to adapt Theorem \ref{thm:almtorkhofin} by replacing the assumption that $R$ is rational of complexity one with the assumption that (the extension of) $\nu$ is Khovanskii-finite with respect to $\bar R$.
\end{remark}

\begin{rem}
	After this manuscript first appeared on the arXiv, Kaveh, Manon, and Murata proved that every projective variety admits a degeneration to a not-necessarily normal variety with a faithful action by a codimension-one torus \cite{degen}. In light of this, Theorem \ref{thm:almtorkhofin} and the question of when varieties with codimension-one torus action admit toric degenerations is all the more relevant.
\end{rem}

In the remainder of this section, we provide two characterizations of homogeneous Khovanskii-finiteness for almost toric varieties. The first characterization has to do with GIT:

Let $R$ be a finitely generated  $M$-graded $\KK$-domain
with $R_0 = \KK$. Then the torus
$H := \Spec(\KK[M])$ acts on the affine variety $X:= \Spec(R)$.
For any weight $w \in \omega(R)$ the 
\emph{set of semistable points with respect to $w$} is
the $H$-invariant open subset
$$X^{ss}(w) := \left\{x \in X\ |\ \ f(x) \neq 0 \text{ for some } f \in R_{nw}, n >0\right\} \subseteq X,$$
and the $H$-action on $X^{ss}(w)$ admits a good quotient
$X^{ss}(w) \rightarrow X^{ss}(w) /\!\!/ H =: Y(w)$.
Moreover
$$Y(w) = \Proj(R(w))=\Proj\left(\bigoplus_{k \geq 0} R_{kw}\right)$$
holds and $Y(w)$ is a projective variety. 
Note that there are only finitely many different
varieties $Y(w)$; see e.g. \cite[Section 3.1]{coxbook} for details.

\begin{corollary}
Let $X$ be an affine almost toric variety with
coordinate ring $R$ such that $R_0 = \KK$ holds, and assume that $\KK$ is uncountable and $\Char(\KK)=0$.
Then $X$ is homogeneously Khovanskii-finite 
if and only if the GIT-quotients $Y(w)$, where $w \in \omega(R)$, are 
all smooth.

\begin{proof}
Let $\nu$ be any homogeneous full rank valuation on $R$.
Then $\nu$ is fully homogeneous 
and we denote the projection 
$C(R, \nu) \rightarrow M_\RR$
with $\pi$, 
where $C(R, \nu)$ is rational polyhedral
due to Theorem \ref{thm:almtorkhofin}.
For each ray $\rho \preceq C(R, \nu)$
we have
$R_{\pi(\rho)} = R(w)$
for some $w \in \omega(R)$. Theorem~\ref{thm:almtorkhofin} and Theorem \ref{thm:khofin} imply that $\nu$ is Khovanskii-finite with respect to $R$ if and only if it is Khovanskii-finite with respect to $R(w)$ for each $w$. But by Corollary \ref{cor:thm3} and Remark \ref{rem:cor}, $R(w)$ is homogeneously Khovanskii-finite if and only if $Y(w)$ is smooth.
\end{proof}
\end{corollary}

The results of \cite{poised} show that \emph{normal} almost toric varieties are homogeneously Khovanskii-finite. The following characterization shows that after taking appropriate Veronese subalgebras, these are in fact the \emph{only} homogeneously Khovanskii-finite almost toric varieties:

\begin{theorem}\label{thm:khofincharacterization}
	Assume that $\KK$ is uncountable and $\Char(\KK)=0$.
An affine almost toric variety $X=\Spec R$ with $R_0 = \KK$ is homogeneously Khovanskii-finite if and only if there exists some $\lambda\in \NN$ such that the Veronese subring
\[
\bigoplus_{u\in M} R_{\lambda\cdot u}
\]
is normal.
\end{theorem}
\begin{proof}
Let $\lambda \in \ZZ_{\geq 0}$ be any positive integer such that
$R' := \bigoplus_{u \in M} R_{\lambda \cdot u}$ is a normal almost toric variety.
Then due to \cite[Theorem 1.3]{poised} $R'$ is homogeneously Khovanskii-finite
and thus $S(R', \nu)$ is finitely generated for every homogeneous valuation $\nu$.
Due to Lemma \ref{lem:fingen} this implies that $\cone(S(R', \nu)) =C(R', \nu) = C(R, \nu)$ is rational polyhedral 
and for each ray $\rho \preceq C(R', \nu)$ there exists a homogeneous element
$f_u \in R'$ such that $\nu(f_u)$ generates $\rho$. As $R' \subseteq R$
we conclude
$\cone(S(R, \nu)) = C(R, \nu)$ and applying Lemma \ref{lem:fingen}
once more the assertion follows.

Assume instead that
$R$ is homogeneously Khovanskii-finite and let $\bar R$ be its normalization.
As discussed in \S\ref{sec:global}, $\bar R$ has the form
\[
\bar R=\bigoplus_{u\in \omega(R)\cap M} H^0(\PP^1,\CO(\lfloor(\D(u)\rfloor)))\cdot \chi^u
\]
for some p-divisor $\D$.
Consider any smooth subdivision $\Sigma$ of the weight cone of $\omega(R) = \omega(\bar{R})$
such that the map 
$$\D \colon \omega(\bar{R}) \rightarrow \WDIV_\RR(\PP^1), \quad u \mapsto \D(u)$$
is linear on each of the cones $\sigma \in \Sigma$. 
As $R$ is homogeneously Khovanskii-finite 
Theorem \ref{thm:khofin} implies that
for each ray $\rho \in \Sigma(1)$ 
the Veronese subalgebra
$R_\rho$ is homogeneously Khovanskii-finite
and due to Theorem \ref{thm:genuskhofincurves} 
there exists
an integer $\lambda_{\rho} \in \NN$ such that
$R_{\rho}(\lambda_{\rho}\cdot \ZZ)$
is normal. Set 
$\lambda := \mu \prod_{\rho \in \Sigma(1)} \lambda_\rho$
with $\mu \geq 1$ such that
$\D(u) \in \WDIV_\ZZ(\PP^1)$ for each
$u \in \rho \cap \lambda M$, where $\rho \in \Sigma(1)$.
We claim that 
$$
R(\lambda M)=
\bar{R}(\lambda M) 
=
\bigoplus_{u\in \lambda M\cap\omega} H^0(\PP^1, \CO(\D(u)))
$$
and thus $R(\lambda M)$ is normal.
Note that by the choice of $\lambda$ we have
$R_u = \bar{R}_u$ for all $u \in \rho\cap \lambda M$ with $\rho \in \Sigma(1)$.
Consider any $u \in \sigma \cap \lambda M$ with $\sigma \in \Sigma$.
Then there exists a unique decomposition
$u = u_1 + \ldots  + u_r$, where $u_i \in \lambda M$ are 
generators of the rays of $\sigma$
and as $\D$ is linear on $\sigma$
we obtain
$\D(u) = \D(u_1)+ \ldots + \D(u_r) \in \WDIV_{\ZZ}(\PP^1)$.
Since the map
$$H^0(\PP^1, \CO_{\PP^1}(D_1)) \times H^0(\PP^1, \CO_{\PP^1}(D_2)) \rightarrow H^0(\PP^1, \CO_{\PP^1}(D_1+D_2))
$$ 
is surjective for any two divisors semiample divisors $D_1, D_2 \in \WDIV_\ZZ(\PP^1)$
we conclude $\bar R_u=R_{u_1}\cdot R_{u_2}\cdots R_{u_r}=R_u$ for 
each $u \in \lambda M$ and the assertion follows.
\end{proof}

\begin{remark}\label{rem:projVSA}
Let $R$ be a $\ZZ$-graded domain with $R_0=\KK$. Then there exists some $\lambda\in \NN$ such that $R(\lambda\ZZ)$ is normal if and only if the projective variety $\Proj(R)$ is normal. Indeed, if $R(\lambda\ZZ)$ is normal, then so is $\Proj(R)=\Proj(R(\lambda\ZZ))$. The converse follows by e.g.~\cite[II Ex 5.14(c)]{hartshorne}.

We apply this in our setting as follows: let $X$ be the affine cone over a projective almost toric variety $Y$. Then Theorem \ref{thm:khofincharacterization} coupled with this remark imply that $X$ is homogeneously Khovanskii-finite if and only if $Y$ is normal.
\end{remark}

\begin{remark}
	It would be interesting and useful to have an effective algorithm for determining if an affine almost toric variety $X$ admits a full rank Khovanskii-finite valuation. Using Theorem \ref{thm:almtorkhofin} together with a description of all possible valuation cones from \cite{poised}, it is possible to reduce to the case of an almost toric surface. Solving Problem \ref{prob:effective} would then deliver such an algorithm.
\end{remark}

\bibliographystyle{alpha}
\bibliography{driver}

\end{document}